\documentclass[11pt,a4paper]{article}
\usepackage[top=3cm, bottom=3cm, left=3cm, right=3cm]{geometry}
\usepackage{tabu}
\usepackage{breqn}
\usepackage[hang,flushmargin]{footmisc} 

\usepackage{amsmath,amsthm,amssymb,graphicx, multicol, array}
\usepackage{enumerate}
\usepackage{enumitem}
\usepackage{hyperref}
\usepackage{setspace}
\usepackage{xcolor}
\usepackage{currfile}

\usepackage{tabto}

\DeclareMathOperator{\li}{li}

\newtheorem{thm}{Theorem}[section]
\newtheorem{lem}{Lemma}[section]

\newtheorem{exe}{Exercise}[section]

\newcommand{\R}{\mathbb{R}}

\usepackage{fancyhdr}
\title{Sum of Divisors Function And The Largest Integer Function Over The Shifted Primes}
\date{}
\author{N. A. Carella}

\begin{document}

\maketitle
\textbf{\textit{Abstract}:} Let $ x\geq 1 $ be a large number, let $ [x]=x-\{x\} $ be the largest integer function, and let $ \sigma(n)$ be the sum of divisors function. This note presents the first proof of the asymptotic formula for the average order $ \sum_{p\leq x}\sigma([x/p])=c_0x\log \log x+O(x) $ over the primes, where $c_0>0$ is a constant. More generally, $ \sum_{p\leq x}\sigma([x/(p+a)])=c_0x\log \log x+O(x) $ for any fixed integer $a$.   
\let\thefootnote\relax\footnote{\today  
}
\tableofcontents

\section{Introduction} \label{s0800}
A series of results have been proved for the finite sums $\sum_{n\leq x}f\left ([x/n]\right )$ over the integers, see \cite{BS2018}, \cite{CN2021}, \cite{ZW2021}, et alii. This note introduces the analytic techniques for evaluating the finite sums $ \sum_{p\leq x}f([x/p]) $ over the primes. This elementary methods ably handle these finite sums for multiplicative functions $f$ defined by Dirichlet convolutions 
$f(n)=\sum_{d\mid n}g(n/d)$, where $f,g: \mathbb{N}\longrightarrow \mathbb{C}$, and the rates of growth, approximately $ f(n)\gg n(\log n)^b $, for some $ b\in \mathbb{Z} $. These elementary methods are efficient, produce very short proofs, and sharp error terms. The asymptotic formula for the fractional sum $ \sum_{p\leq x}\varphi([x/p]) $ is assembled in \cite{CN2021B}, and the asymptotic formula for the fractional sum $ \sum_{p\leq x}\sigma([x/p]) $ of the sum of divisors function $  \sigma$ in Theorem \ref{thm0505.010}, is evaluated here. This is a new result in the literature.

\section{Sum of Divisors Function Over The Primes}\label{S0505}
The result in this section deals with the sum of divisors function $ \sigma(n)=n\sum_{d\mid n}1/d $ composed with the largest integer function $[z]=z-\{z\}$. The function $\sigma: \mathbb{N}\longrightarrow \mathbb{N}$ is multiplicative and satisfies the growth condition $ \sigma(n)\gg n $. The first asymptotic formula for the fractional sum of divisor function over the primes is given below.
\begin{thm}\label{thm0505.010} If $ x\geq 1 $ is a large number, then, 
\begin{equation}\label{eq0505.010}
\sum_{p\leq x}\sigma\left (\left [\frac{x}{p}\right ]\right )= c_{0}x\log \log x+c_{1}x+O\left (\li(x)\log \log x\right )\nonumber,
\end{equation}
where $c_{0}=\zeta(2)$, and $c_{1}=B_1\zeta(2)$ are constants.
\end{thm}
\begin{proof} Use the identity $ \sigma(n)=n\sum_{d\mid n}1/d $ to rewrite the finite sum, and switch the order of summation:
\begin{eqnarray}\label{eq0505.020}
\sum_{p\leq x}\sigma\left (\left [\frac{x}{p}\right ]\right )
&=& \sum_{p\leq x} \left [\frac{x}{p}\right]\sum_{d\,\mid\, [x/p]}\frac{1}{d} \\
&=& \sum_{d\leq x} \frac{1}{d}\sum_{\substack{p\leq x\\d\,\mid\, [x/p]}}\left [\frac{x}{p}\right] \nonumber.
\end{eqnarray}
Apply Lemma \ref{lem0802.050} to remove the congruence on the inner sum index, and break it up into two subsums. Specifically,
\begin{eqnarray}\label{eq0505.030}
\sum_{d\leq x} \frac{1}{d}\sum_{\substack{p\leq x\\d\,\mid\, [x/p]}}\left [\frac{x}{p}\right]
&=&\sum_{d\leq x} \frac{1}{d}\sum_{p\leq x}\left [\frac{x}{p}\right]\cdot \frac{1}{d}\sum_{0\leq a\leq d-1}e^{i2\pi a[x/p]/d} \\
&=&\sum_{d\leq x} \frac{1}{d^2}\sum_{p\leq x}\left [\frac{x}{p}\right]+\sum_{d\leq x} \frac{1}{d^2}\sum_{p\leq x}\left [\frac{x}{p}\right]\sum_{0< a\leq d-1}e^{i2\pi a[x/p]/d}\nonumber \\
&=&S(x)\quad +\quad T(x)\nonumber.
\end{eqnarray}
The first sum $S(x)$ is computed in Lemma \ref{lem0507.100} and the sum $T(x)$ is computed in Lemma \ref{lem0509.200}. Summing these expressions 
\begin{eqnarray}
\sum_{p\leq x}\sigma\left (\left [\frac{x}{p}\right ]\right )
&=&S(x)+T(x)\\
&=&c_{0}x\log \log x+c_{1}x+c_{2}\li(x)+O\left (xe^{-c\sqrt{\log x}}\right )+O\left (\li(x)\log \log x\right )\nonumber\\
&=&c_{0}x\log \log x+c_{1}x+O\left (\li(x)\log \log x\right )\nonumber,
\end{eqnarray}
where $c_0=\zeta(2)$, and $c_1=B_1\zeta(2)$, $c_2=(1-\gamma)\zeta(2)$ are constants, and $c>0$ is an absolute constant.
\end{proof}
The constants occurring in the above expression are the followings.
\begin{enumerate}
\item $\displaystyle \zeta(2)=\frac{\pi^2}{6}=1.644934066848226436472415\ldots$, the zeta constant,
\item $\displaystyle B_1=\lim_{x\to \infty}\frac{1}{x}\sum_{p\leq x}\frac{1}{p}-\log \log x=0.261497212847642783755426\ldots $, Mertens constant,
\item $\displaystyle\gamma=\lim_{x\to \infty}\frac{1}{x}\sum_{n\leq x}\frac{1}{n}- \log x=0.577215664901532860606512\ldots$, Euler constant,
\item $\displaystyle c_1=B_1\zeta(2)=0.430145673798949331799597\ldots$, 
\item $\displaystyle c_2=(1-\gamma)\zeta(2)=0.695452355733244911926851\ldots$.
\end{enumerate} 

The standard proof for the average order over the primes
\begin{equation}\label{eq0505.040}
\sum_{p\leq x}\sigma(p)= c_3\li(x^2)+O\left (  x^2 e^{-c\sqrt{\log x}}\right ),
\end{equation}
where $c>0$ is an absolute constant, this follows from the prime number theorem, \cite[Theorem 6.9]{MV2007}, and partial summation. And the standard proof for the average order over the shifted primes
\begin{equation}\label{eq0505.050}
\sum_{p\leq x}\sigma(p-1)= \frac{315 \zeta(3)}{2\pi^4}x+O\left (  \frac{x}{(\log x)^{0.999}}\right )
\end{equation}
was proved in \cite{LY1961}.

\section{The Sum $S(x)$}\label{S0507}
The detailed and elementary evaluation of the asymptotic formula for the finite sum $S(x)$ occurring in \eqref{eq0505.030} are recorded in this section. Both conditional and unconditional results are provided. For a real number $z\in \R$, the largest integer function is defined by $[z]=z-\{z\}$ 

\begin{lem}\label{lem0507.100} Let $ x\geq 1 $ be a large number. Then,
\begin{equation}\label{eq0507.110}
\sum_{d\leq x} \frac{1}{d^2}\sum_{p\leq x}\left [\frac{x}{p}\right]
= c_{0}x\log \log x+c_{1}x+c_{2}\li(x)+O\left (xe^{-c\sqrt{\log x}}\right ),
\end{equation}
where $c_{0}=\zeta(2)$, and $c_{1}=B_1\zeta(2)$, $c_{2}=(1-\gamma)\zeta(2)$ are constants, and $c>0$ is an absolute constant.
\end{lem}
\begin{proof}Expand the bracket and evaluate the two subsums. Specifically,
\begin{eqnarray}\label{eq0507.120}
S(x)&=&\sum_{d\leq x} \frac{1}{d^2}\sum_{n\leq x}\left [\frac{x}{p}\right]\\
&=&x\sum_{d\leq x} \frac{1}{d^2}\sum_{p\leq x}\frac{1}{p}-\sum_{d\leq x} \frac{1}{d^2}\sum_{p\leq x}\left \{\frac{x}{p}\right\}\nonumber\\
&=&S_0(x)-S_1(x)\nonumber. 
\end{eqnarray}

Substituting the standard asymptotic for the zeta constant $\sum_{n\leq x} 1/n^2$, and the prime harmonic sum $\sum_{p\leq x} 1/p$, return the followings.
\begin{eqnarray}\label{eq0507.130}
S_0(x)&=&x\sum_{d\leq x} \frac{1}{d^2}\sum_{p\leq x}\frac{1}{p}\\
&=&x\left ( \frac{1}{\zeta(2)}+O\left (\frac{1}{x}\right )\right )\left ( \log \log x+B_1+O\left (e^{-c\sqrt{\log x}}\right )\right )\nonumber\\
&=& c_{0}x\log \log x+c_{1} x  +O\left (xe^{-c\sqrt{\log x}}\right )\nonumber,
\end{eqnarray} 
where $ B_1>0 $ is Mertens constant, $c>0$ is an absolute constant, $c_{0}=\zeta(2)$, and $c_{1}=B_1\zeta(2)$.\\

Likewise, substituting the standard asymptotics for the zeta constant $\sum_{n\leq x} 1/n^2$, and the prime number theorem for fractional parts sum $\sum_{p\leq x} [x/p]$, see \cite{PF2010}, \cite[Exercise 1g, p.\ 248]{MV2007}, return the followings.
\begin{eqnarray}\label{eq0507.140}
S_1(x)&=&\sum_{d\leq x} \frac{1}{d^2}\sum_{p\leq x}\left \{\frac{x}{p}\right\}\\
&=& \left ( \frac{1}{\zeta(2)}+O\left (\frac{1}{x}\right )\right )\left ( (1-\gamma)\li(x)+O\left (xe^{-c\sqrt{\log x}}\right )\right )\nonumber\\
&=& c_{2}\li(x)+O\left (xe^{-c\sqrt{\log x}}\right )\nonumber,
\end{eqnarray} 
where $\li(x)$ is the logarithm integral, $\gamma$ is the Euler constant, $c_{2}=(1-\gamma)\zeta(2)$, and $c>0$ is an absolute constant.
Subtracting the expressions \eqref{eq0507.130} and \eqref{eq0507.140} yields
\begin{equation}\label{eq0507.150} 
S_(x)=S_0(x)-S_1(x)=c_{0}x\log \log x+c_{1} x - c_{0}\li(x)+O\left (xe^{-c\sqrt{\log x}}\right ),
\end{equation}
where $c_{0}=\zeta(2)$, and $c_{1}=B_1\zeta(2)$, $c_{2}=(1-\gamma)\zeta(2)$, and $c>0$ is an absolute constant.
\end{proof}

\begin{lem}\label{lem0507.250} Assume the RH. Let $ x\geq 1 $ be a large number. Then,
\begin{equation}\label{eq0507.210}
\sum_{d\leq x} \frac{1}{d^2}\sum_{p\leq x}\left [\frac{x}{p}\right]
= c_{0}x\log \log x+c_{1}x+c_{2}\li(x)+O\left (x(\log x)^2\right ),
\end{equation}
where $c_{0}=\zeta(2)$, and $c_{1}=B_1\zeta(2)$, $c_{2}=(1-\gamma)\zeta(2)$ are constants.
\end{lem}

\begin{proof} Everything remain the same as the previous proof. However, the unconditional error term is replaced with the conditional error term.
\end{proof}

\section{The Sum $T(x)$}\label{S0509}
The detailed evaluation of the asymptotic formula for the finite sum $T(x)$ occurring in \eqref{eq0505.030} are recorded in this section. For a real number $z\in \R$, the fractional function is defined by $\{z\}=z-[z]$.
\begin{lem}\label{lem0509.200} Let $ x\geq 1 $ be a large number. Then,
\begin{equation}\label{eq0509.210}
\sum_{d\leq x} \frac{1}{d^2}\sum_{p\leq x}\left [\frac{x}{p}\right]\sum_{0< a\leq d-1}e^{i2\pi a[x/p]/d}
=O\left (\li(x)\log \log x\right )\nonumber.
\end{equation}
\end{lem}
\begin{proof} Let $\pi(x)=\#\{\text{prime } p\leq x\}$ be the primes counting function, let $\li(x)$ be the logarithm integral, and let $p_k$ be the $k$th prime in increasing order. The sequence of values
\begin{equation}\label{eq0509.225}
\left [\frac{x}{p_{k}}\right]=\left [\frac{x}{p_{k+1}}\right]=\cdots =\left [\frac{x}{p_{k+r}}\right]
\end{equation}
arises from the sequence of primes $x/(n+1)\leq p_{k}, p_{k+1}, \ldots, p_{k+r}\leq x/n$. Therefore, the value $m=[x/p]\geq1$ is repeated 
\begin{equation}\label{eq0509.220}
\pi\left (\left [\frac{x}{n}\right]\right)-\pi\left (\left [\frac{x}{n+1}\right]\right)=\frac{\li(x)}{n(n+1)} +O\left (\frac{x}{n}e^{-c\sqrt{\log x}}\right ) 
\end{equation}
times as $p$ ranges over the prime values in the interval $[x/(n+1), x/n]$, see Exercise \ref{exe0303.100} in Section \ref{S0303}. Hence, substituting \eqref{eq0509.220} into the triple sum $T(x)$, and reordering it yield
\begin{eqnarray}\label{eq0509.230}
T(x)&=&\sum_{p\leq x}\left [\frac{x}{p}\right]\sum_{d\leq x} \frac{1}{d^2}\sum_{0< a\leq d-1}e^{i2\pi m/d}\\
&=&\sum_{p\leq x}\left ( \frac{\li(x)}{n(n+1)}+O\left (\frac{x}{n}e^{-c\sqrt{\log x}}\right )\right)\sum_{d\leq x} \frac{1}{d^2}\sum_{0< a\leq d-1}e^{i2\pi am/d}\nonumber\\
&=&\li(x)\sum_{n\leq x}\frac{1}{n(n+1)}\sum_{d\leq x} \frac{1}{d^2}\sum_{0< a\leq d-1}e^{i2\pi am/d}\nonumber\\
&& \hskip 1.5 in + O\left (xe^{-c\sqrt{\log x}}\sum_{n\leq x} \frac{1}{n}\sum_{d\leq x} \frac{1}{d^2}\sum_{0< a\leq d-1}e^{i2\pi am/d}\right)\nonumber\\ 
&=&T_{0}(x)+T_{1}(x)\nonumber.
\end{eqnarray}
The finite subsums $T_{0}(x)$ computed in Lemma \ref{lem0509.300}, and $T_{1}(x)$, computed in Lemma \ref{lem0509.400}. Summing yields 
\begin{eqnarray}\label{eq0509.240}
T(x)&=&T_{0}(x)+T_{1}(x)= O\left (\li(x)\log \log x\right )+O\left (xe^{-c\sqrt{\log x}}\right )\\
&=&O\left (\li(x)\log \log x\right )\nonumber,
\end{eqnarray}
where $c>0$ is an absolute constant.
\end{proof}

\subsection{The Sum $T_{0}(x)$}
\begin{lem}\label{lem0509.300} Let $ x\geq 1 $ be a large number. Then,
\begin{equation}\label{eq0507.310}
\li( x)\sum_{n\leq x}\frac{1}{n(n+1)}\sum_{d\leq x} \frac{1}{d^2}\sum_{0< a\leq d-1}e^{i2\pi am/d}
= O\left (\li(x)\log \log x\right )\nonumber.
\end{equation}
\end{lem}
\begin{proof} Partition the finite sum into two finite subsums 
\begin{eqnarray}\label{eq0509.320}
T_{0}(x)
&=& \li( x)\sum_{n\leq x}\frac{1}{n(n+1)}\sum_{\substack{d\leq x\\d\mid m}} \frac{1}{d^2}\sum_{0< a\leq d-1}e^{i2\pi am/d}\\\
&&\hskip 1.5 in +\li( x)\sum_{n\leq x}\frac{1}{n(n+1)}\sum_{\substack{d\leq x\\d\nmid m}} \frac{1}{d^2}\sum_{0< a\leq d-1}e^{i2\pi am/d}\nonumber\\
&=&T_{00}(x)+ T_{01}(x)\nonumber.
\end{eqnarray}
The finite subsums $T_{00}(x)$ is estimated in Lemma \ref{lem0509.500}, and $T_{01}(x)$ is estimated in Lemma \ref{lem0509.600}. These finite sums correspond to the subsets of integers $p\leq x$ such that $d\mid m$, and $d\nmid m$, respectively. Summing yields 
\begin{equation}\label{eq0509.340}
T_0(x)=T_{00}(x)+T_{01}(x)= O\left (\li(x)\log \log x\right ).
\end{equation}

\end{proof}

\subsection{The Sum $T_{00}(x)$}
\begin{lem}\label{lem0509.500} Let $ x\geq 1 $ be a large number, let $[x]=x-\{x\}$ be the largest integer function, and $m=[x/p]\leq [x/n]\leq x$. Then,
\begin{equation}\label{eq0509.510}
\li( x)\sum_{n\leq x}\frac{1}{n(n+1)}\sum_{d\leq x} \frac{1}{d^2}\sum_{0< a\leq d-1}e^{i2\pi am/d}
= O(\li( x)\log \log x).
\end{equation}
\end{lem}
\begin{proof} The set of values $m=[x/p]\leq [x/n]\leq x$ such that $d\mid m$. Evaluating the incomplete function returns
\begin{eqnarray}\label{eq0509.520}
T_{00}(x)&=&\li( x)\sum_{n\leq x}\frac{1}{n(n+1)}\sum_{d\leq x} \frac{1}{d^2}\sum_{0< a\leq d-1}e^{i2\pi am/d}\\
&=& \li( x)\sum_{n\leq x}\frac{1}{n(n+1)}\sum_{\substack{d\leq x\\d\mid m}} \frac{1}{d^2}\cdot (d-1)\nonumber\\
&=&\li( x)\sum_{n\leq x}\frac{1}{n(n+1)}\sum_{\substack{d\leq x\\d\mid m}} \frac{1}{d}-\li( x)\sum_{n\leq x}\frac{1}{n(n+1)}\sum_{\substack{d\leq x\\d\mid m}} \frac{1}{d^2}\nonumber\\
&=&T_{20}(x)+T_{21}(x)\nonumber.
\end{eqnarray}
The first term has the upper bound
\begin{equation}\label{eq0509.530}
T_{20}(x)=\li( x)\sum_{n\leq x}\frac{1}{n(n+1)}\sum_{\substack{d\leq x\\d\mid m}} \frac{1}{d}
=\li( x)\sum_{n\leq x}\frac{1}{n(n+1)}\sum_{d\mid m} \frac{1}{d} \ll x\log \log x.
\end{equation}
This follows from upper bound of the sum of divisors function
\begin{equation}\label{eq0509.540}
\sum_{\substack{d\leq x\\d\mid m}} \frac{1}{d}=\sum_{d\mid m} \frac{1}{d}= \frac{\sigma(m)}{m}\leq 2\log \log x,
\end{equation}
where $m=[x/p] \leq [x/n]\leq x$. The second term has the upper bound
\begin{equation}\label{eq0509.560}
T_{21}(x)=\li( x)\sum_{n\leq x}\frac{1}{n(n+1)}\sum_{\substack{d\leq x\\d\mid m}} \frac{1}{d^2}
=\li( x)\sum_{n\leq x}\frac{1}{n(n+1)}\sum_{d\mid [x/n]} \frac{1}{d^2}
\ll \li( x).
\end{equation}
Summing yields $T_{00}(x)=T_{20}(x)+T_{21}(x)=O(\li( x)\log \log x)$
\end{proof}

\subsection{The Sum $T_{01}(x)$}
\begin{lem}\label{lem0509.600} Let $ x\geq 1 $ be a large number. Then,
\begin{equation}\label{eq0509.610}
\li( x)\sum_{n\leq x}\frac{1}{n(n+1)}\sum_{\substack{d\leq x\\d\nmid m}} \frac{1}{d^2}\sum_{0< a\leq d-1}e^{i2\pi am/d}
= O\left (\li( x)\right)\nonumber.
\end{equation}
\end{lem}
\begin{proof} The set of values $m=[x/p]\leq [x/n]\leq x$ such that $d\nmid m$. Evaluating the incomplete indicator function returns
\begin{eqnarray}\label{eq0509.645}
T_{1}(x)&=&\li( x)\sum_{n\leq x}\frac{1}{n(n+1)}\sum_{\substack{d\leq x\\d\nmid m}} \frac{1}{d^2}\sum_{0< a\leq d-1}e^{i2\pi am/d}\\
&=&\li( x)\sum_{n\leq x}\frac{1}{n(n+1)}\sum_{\substack{d\leq x\\d\nmid m}} \frac{1}{d^2}\cdot (-1)\nonumber\\
&=&\li( x)\sum_{n\leq x}\frac{1}{n(n+1)}\left (c_{1}(n)+O(1)\right)\nonumber\\
&=&O\left (\li( x)\right)\nonumber,
\end{eqnarray}
where $|c_{1}(n)|<2$ depends on $n$. 
\end{proof}

\subsection{The Sum $T_{1}(x)$}
\begin{lem}\label{lem0509.400} If $ x\geq 1 $ is a large number, then,
\begin{equation}\label{eq0509.410}
xe^{-c_0\sqrt{\log x}}\sum_{n\leq x} \frac{1}{n}\sum_{d\leq x} \frac{1}{d^2}\sum_{0< a\leq d-1}e^{i2\pi am/d}
= O\left (xe^{-c\sqrt{\log x}}\right )\nonumber,
\end{equation}
where $c_0>0$ and $c>0$ are absolute constants.
\end{lem}

\begin{proof} The absolute value provides an upper bound:
\begin{eqnarray}\label{eq0509.420}
\left |T_{1}(x)\right |&=&\left | xe^{-c_0\sqrt{\log x}}\sum_{n\leq x} \frac{1}{n}\sum_{d\leq x} \frac{1}{d^2}\sum_{0< a\leq d-1}e^{i2\pi am/d}\right |\\
&\leq & xe^{-c_0\sqrt{\log x}}\sum_{n\leq x} \frac{1}{n}\sum_{d\leq x} \frac{1}{d}\nonumber\\
&=&O\left ((x\log ^2 x)e^{-c_0\sqrt{\log x}}\right )\nonumber\\
&=&O\left (xe^{-c\sqrt{\log x}}\right )\nonumber,
\end{eqnarray}
where $c_0>0$ and $c>0$ are absolute constants.
\end{proof}

\begin{lem}\label{lem0802.050} Let $ x\geq 1 $ be a large number, and let $1\leq d,m,n\leq x$ be integers. Then,
\begin{equation}\label{eq0802.010}
\frac{1}{d}\sum_{0\leq a\leq d-1}e^{i2\pi am/d} =\left \{\begin{array}{ll}
1 & \text{ if } d\mid m,  \\
0& \text{ if } d\nmid m, \\
\end{array} \right .
\end{equation}
\end{lem}

\section{Numerical Data}\label{S3507}
Small numerical tables were generated by an online computer algebra system to estimate the constant, the range of numbers $ x\leq 10^6
 $ is limited by the wi-fi bandwidth. The error term is defined by
\begin{equation}\label{eq3507.033}
E(x)=\sum_{p\leq x}\sigma([x/p])-c_0x\log \log x -c_1 x,
\end{equation}
where $c_{0}=\zeta(2)$, and $c_{1}=B_1\zeta(2)$ are constants.
\begin{table}[h!]
\centering
\caption{Numerical Data For $\sum_{p\leq x}\sigma([x/p])$.} \label{t3507.003}
\begin{tabular}{l|l|l|r| r}
$x$&$\pi(x)$&$\sum_{p\leq x}\sigma([x/p])$&$c_0x\log \log x+c_1x$&Error $E(x)$\\
\hline
10&$4$&$14$&   $18.02$   &$4.02$\\
100&$25$&$277$&  $294.22$   &$17.22$\\
1000&$168$&$3852$&   $3609.21$   &$-242.79$\\
10000&$1229$&$45843$&  $40824.25$   &$-5018.74$\\
100000 &$9592$& $481903$ &   $444948.14$   &$-36954.86$\\
1000000&$78498$&$5412077$&   $4749388.38$   &$-662688.62$\\
\end{tabular}
\end{table}

\section{Problems} \label{S0303}
\begin{exe}\label{exe0303.100} {\normalfont Let $\pi(x)=\#\{\text{prime } p\leq x\}$, and let $\li(x)$ be the logarithm integral. Show that
$$\pi\left (\left [\frac{x}{n}\right]\right)-\pi\left (\left [\frac{x}{n+1}\right]\right)=\frac{\li(x)}{n(n+1)} +O\left (\frac{x}{n}e^{-c\sqrt{\log x}}\right ) ,$$
where $c>0$ is an absolute constant.
}
\end{exe}

\begin{exe}\label{exe0303.120} {\normalfont Sharpen the proof of Lemma \ref{lem0507.500}. Specifically, show that
$$\li( x)\sum_{n\leq x}\frac{1}{n(n+1)}\sum_{d\leq x} \frac{1}{d^2}\sum_{0< a\leq d-1}e^{i2\pi am/d}
= a_0\li( x)\log \log x+a_1\li( x)+O\left (xe^{-c\sqrt{\log x}}\right ) ,$$
where $\li(x)$ be the logarithm integral, $a_0\ne0$ and $a_1\ne0$ are constants, and $c>0$ is an absolute constant.
}
\end{exe}

\begin{exe}\label{exe0303.200} {\normalfont Let $n\geq1$ be an integer. Determine the closed form evaluations of the finite sums
$$\sum_{d\mid n} \frac{1}{d^2}=?,\quad \sum_{d\mid n} \frac{1}{d^3}=?, \quad \sum_{d\mid n} \frac{1}{d^4}=?,\quad\sum_{d\mid n} \frac{1}{d^5}=?,\ldots,$$
Hint: Consider the generalized sum of divisors function $\sigma_s(n)=\sum_{d\mid n}d^s$.
}
\end{exe}

\begin{exe}\label{exe0303.200} {\normalfont Let $n\geq1$ be an integer, and let $\mu$ be the Mobius function. Determine the closed form evaluations of the finite sums
$$\sum_{d\mid n} \frac{\mu(d)}{d^2}=?,\quad \sum_{d\mid n} \frac{\mu(d)}{d^3}=?, \quad \sum_{d\mid n} \frac{\mu(d)}{d^4}=?,\quad\sum_{d\mid n} \frac{\mu(d)}{d^5}=?,\ldots,$$
Hint: Consider the (Jordan function) generalized Euler function $\varphi_s(n)=\prod_{p\mid n}(1-p^{-s})$.
}
\end{exe}


\currfilename.\\


\begin{thebibliography}{998}
\bibitem{AP1976} Apostol, Tom M. \textit{\color{blue}Introduction to analytic number theory}. Undergraduate Texts in Mathematics. Springer-Verlag, New York-Heidelberg, 1976.

\bibitem{BS2018} Olivier Bordelles, Randell Heyman, Igor E. Shparlinski. \textit{\color{blue}On a sum involving the Euler function}, http://arxiv.org/abs/1808.00188.

\bibitem{CN2021} Carella, N. A. \textit{\color{blue}Average Orders of the Euler Phi Function, The Dedekind Psi Function, The Sum of Divisors Function, And The Largest Integer Function}. http://arxiv.org/abs/2101.02248. 

\bibitem{CN2021B} Carella, N. A. \textit{\color{blue}Euler Totient Function Over The Shifted Primes}. http://arxiv.org/abs/2105.00790.
 


\bibitem{LY1961} Linnik, Y. V. \textit{\color{blue}The dispersion method in binary additive problems}. Izdat. Leningrad Univ., Leningrad, 1961.


\bibitem{MV2007} Montgomery, Hugh L.; Vaughan, Robert C. \textit{\color{blue}Multiplicative number theory. I. Classical theory.} Cambridge University Press, Cambridge, 2007.

\bibitem{PF2010} Pillichshammer, Friedrich. \textit{\color{blue}Euler's constant and averages of fractional parts}. Amer. Math. Monthly 117 (2010), no. 1, 78-83.


\bibitem{ZW2021} Zhao, F. Wu, J. \textit{\color{blue}On a sum involving the sum of divisors function.} J. Math. Art. ID 5574465, 7 pp.       
\end{thebibliography}
\end{document}